\documentclass[12 pt,oneside,reqno,a4paper]{amsart} 
\usepackage{amsfonts,amssymb,amscd,amsmath,enumerate,verbatim,calc} 
\usepackage{float}
\usepackage{amsthm}
\newtheorem{theorem}{Theorem}[section]
\newtheorem{definition}{Definition}[section]

\newtheorem{lemma}[theorem]{Lemma}

\newtheorem{proposition}[theorem]{Proposition}
\newtheorem{remark}[theorem]{Remark}
\usepackage{mathtools}

\newtheorem{example}[theorem]{Example}
\numberwithin{equation}{section}

\usepackage{amsfonts}
\newcommand{\field}[1]{\mathbb{#1}}

                   \newcommand{\Z}{\field{Z}}

\begin{document}
	\title{multiplicative lie algebra structures on semi-direct product of groups}
	
	\author{Deepak Pal$^{1}$, Amit Kumar$^{2}$, Sumit Kumar Upadhyay$^{3}$ AND Seema Kushwaha$^{4}$\vspace{.4cm}\\
		{Department of Applied Sciences,\\ Indian Institute of Information Technology Allahabad\\Prayagraj, U. P., India} }

\begin{abstract}
The main aim of this paper is to determine the multiplicative lie algebra structures on the semi-direct product of an abelian group with a group under certain conditions.
\end{abstract}

	\maketitle

	\section{Introduction}
A multiplicative Lie algebra structure on a group $G$ is a function that satisfies the identities similar to the universal identities of the commutator function. We know that if $G$ is a non-cyclic group, then there are always at least two distinct multiplicative Lie algebra structures on $G$. Also, if $G = \mathbb{Z}_p^n $, where $p$ is a prime, then distinct multiplicative Lie algebra structures on $G$ can be determined by the classification problem of $n-$dimensional Lie algebra over the field $\mathbb{Z}_p$. Thus, the following are interesting problems: 

\textbf{Problem 1:} How many distinct (up to isomorphism) multiplicative Lie algebra structures exist on the group $G?$ 

\textbf{Problem 2:} Let $H$ be a subgroup of $G$ with a multiplicative Lie algebra structure $\star$ on $H$. Can we define a multiplicative Lie algebra structures $\tilde{\star}$ on $G$ with the help of $\star?$

In 2019, Walls (\cite{GLW}) investigated the construction of a multiplicative Lie algebra structure on $G$ (for details, one can see Theorem 3.7 of \cite{GLW}).  In \cite{MS}, Pandey and Upadhyay discussed Problem 1 and gave a precise characterization
of the group homomorphisms from the exterior square $G\wedge G$ to $G$ which determine a multiplicative Lie
algebra structure on $G$. They also found the number of distinct (up to isomorphism) multiplicative Lie algebra structures on some classes of finite groups like $D_n, Q_n $, etc. 

The main aim of this paper is to determine the multiplicative Lie algebra structures on the semi-direct product of groups motivated by Problem 2. More precisely, let $H$ be an abelian group with trivial multiplicative Lie algebra structure and $K$ be a multiplicative Lie algebra. Then with the help of multiplicative Lie algebra structure on $K$,  we define multiplicative Lie algebra structures on the semi-direct product $G$ of $H$ and $K$ such that $H$ is an ideal of $G$. In particular, 
% For example, suppose $H$ is a cyclic subgroup of a non-cyclic group $G$, then we can not get improper multiplicative Lie algebra structure on $G$ with the help of multiplicative Lie algebra structure on $H$ which is only trivial. In this paper, we determine multiplicative Lie algebra structures on the semi-direct product of groups under certain conditions  
if $G = H \times K$ and $(|H|, |K|) = 1$, then we see that every multiplicative Lie algebra structure on $G$ is completely determined by multiplicative Lie algebra structure on $K$. This method will help to determine all distinct multiplicative Lie algebra structures on a given group.  

%At the end of the article, we give a table having the Schur multiplier of some finite multiplicative Lie algebras.

	Now, we give few definitions and results  which are useful for the article.
	
\begin{definition}
A short exact sequence $$ {1}\longrightarrow  H\overset{\alpha} \longrightarrow G\overset{\beta} \longrightarrow  K \longrightarrow {1}$$ 
of multiplicative Lie algebras is called an extension of $H$ by $K.$ A map $t : K \to G$ is called a section of extension if $\beta \circ t = I_K$ and $t(1) = 1.$ 
\end{definition}
\begin{remark}[Proposition 3.4 (\cite{RLS})]
Let $H$ be an abelian group and $End(H)$ be the set of all group endomorphisms on $H$. Then $(End(H),\cdot, *)$ is a multiplicative Lie algebra, where $(F_1\cdot F_2)(h) = F_1(h)F_2(h)$ and $(F_1*F_2)(h)=F_1(F_2(h))F_2(F_1(h^{-1}))$.
\end{remark}
	 
%\begin{remark}[Proposition 3.4 (\cite{RLS})] \label{Schur}
%Let $K$ be a finite group with multiplicative Lie algebra structure $\star$. Then we have a short exact sequence 
%$1 \longrightarrow  Hom\bigg(\frac{\wedge^2K}{J},\mathbb{C}^*\bigg) \longrightarrow  \tilde M(K)\longrightarrow M(K)$, where $J$ is the subgroup of $\wedge^2K$ generated by $\{((x\star y) \wedge ^yz)((y\star z) \wedge ^zx)((z\star x) \wedge ^xy)\mid x, y, z \in K\}$.
%\end{remark}
	 
\section{Multiplicative lie algebra structures on semi-direct product}		
		Consider an extension $	{1}\longrightarrow  H\overset{i} \longrightarrow G\overset{\beta} \longrightarrow  K \longrightarrow {1}$ of $H$ by $K,$  where $H$ is an abelian group with trivial multiplicative Lie algebra structure and $K$ is a group with multiplicative Lie algebra structure $\star$. Let $t : K \to G $ be a section. Then by Remark $4.14$ (\cite{MS1}), the group operation $``\cdot"$ and the multiplicative Lie product $\tilde{\star}$ in $G$ are given by 
		\begin{align*}
	&	ht(x)\cdot kt(y)=h\sigma_x^t(k)f^t(x,y)t(xy)\\ & ht(x)\tilde{\star} kt(y)=hk\Gamma^t_x(k)\sigma_{(x\star y)}(h^{-1}k^{-1}\Gamma^t_y(h^{-1})) h^t(x,y)t(x\star y), 
		\end{align*}
		where $\sigma_x^t(k)=t(x)kt(x)^{-1}$, $\Gamma^t_x(k)=t(x)\star k$ are group homomorphisms on $H$
and $f^t, h^t: K\times K \longrightarrow H$ are maps satisfying the following identities
\begin{enumerate}
\item $f^t(1,x) = f^t(x,1) = 1$ and $f^t(x,y)f^t(xy, z) = \sigma_x^t(f^t(y,z))f^t(x, yz)$;
\item $ h^t(x,1)=h^t(1,x)=h^t(x,x)=1$.
\end{enumerate} 
		
In fact, we have a group homomorphism $\sigma^t : K \to Aut(H)$	defined by $\sigma^t(x) = \sigma^t_x$ and a map $\Gamma : K \to End(H)$	defined by $\Gamma^t(x) = \Gamma^t_x$. 
\begin{proposition}
The maps $\sigma^t$ and $\Gamma^t$ are independent on the choice of section $t$.
\end{proposition}
\begin{proof}
Let $s$ and $t$ are two sections. Then there exists a map $g: K \to H$ with $g(1) = 1$ such that $s(x) = g(x)t(x)$ for every $x\in K$.

Now, $\sigma_{x}^s(h) = s(x)hs(x)^{-1} = g(x)t(x)ht(x)^{-1}g(x)^{-1}= g(x)\sigma_x^t(h)g(x)^{-1} = \sigma_x^t(h)$ (since $H$ is abelian). This shows that the group homomorphism $\sigma^t : K \to Aut(H)$ is independent on the choice of section $t$.

Also, $\Gamma_{x}^s(h) = s(x)\star h = (g(x)t(x))\star h=~ ^{g(x)}(t(x))\star h) (g(x))\star h) = \Gamma_x^t(h)$ (since $H$ is abelian with trivial multiplicative Lie algebra structure). This shows that the map $\Gamma^t : K \to End(H)$ is independent on the choice of section $t$.
\end{proof}
So, now onwards we denote $\sigma^t$ and $\Gamma^t$ by $\sigma$ and $\Gamma$, respectively. Suppose $t$ is a group homomorphism, that is, $G \cong H \rtimes_{\sigma} K$. Then $f^t(x,y) = 1$, for all $x, y \in K$ and $ht(x)\cdot kt(y)=h\sigma_x(k)t(xy)$.
\begin{proposition} \label{Gamma} If $t$  is a splitting,  then we have
$\Gamma_{xy}(h)=\Gamma_x(h)\sigma_x (\Gamma_y(h)) $ and $\Gamma_{x\star y}(\sigma_y(h))=\Gamma_x (\Gamma_y(h)) \Gamma_{xyx^{-1}} (\Gamma_x(h^{-1})) $, for all $x, y \in K$ and $h \in H$.
\end{proposition}
\begin{proof}
Since $t$ is a group homomorphism, we have
\begin{align*}
\Gamma_{xy}(h)&=t(xy)\star h \\&= (t(x)t(y))\star h\\
&= ^{t(x)}(t(y)\star h)(t(x)\star h)\\
&= \Gamma_x(h) ^{t(x)}(\Gamma_y(h)) \\&= \Gamma_x(h)\sigma_x (\Gamma_y(h)) 
\end{align*}
Now, $\Gamma_{x\star y}(h) = t(x\star y)\star h = (h(x,y)^{-1} (t(x)\star t(y)))\star h$ \vspace{.2cm}

$=^{{h(x,y)}^{-1}}((t(x)\star t(y))\star h) ({h(x,y)}^{-1}\star h) = (t(x)\star t(y))\star h. $\vspace{.2cm}

Since \vspace{.2cm} 

$ ((t(x)\star t(y)) \star ^{t(y)}h)((t(y)\star h) \star ^ht(x)) ((h\star t(x)) \star ^{t(x)}t(y)) = 1 $, we have\vspace{.2cm}

$((t(x)\star t(y)) \star \sigma_y(h)) (\Gamma_{y}(h) \star ^ht(x))(\Gamma_x(h^{-1})\star t(xyx^{-1})) = 1$\vspace{.2cm}

$ \implies ((\Gamma_{(x\star y)} \star \sigma_y(h)) (\Gamma_{y}(h) \star t(x))(\Gamma_x(h^{-1})\star t(xyx^{-1})) = 1  $\vspace{.2cm}

$ \implies \Gamma_{(x\star y)} (\sigma_y(h) \Gamma_x (\Gamma_{y}(h^{-1})) \Gamma_{xyx^{-1}}(\Gamma_x(h)) = 1 $\vspace{.2cm}

$ \implies \Gamma_{(x\star y)} (\sigma_y(h) = \Gamma_x (\Gamma_{y}(h)) \Gamma_{xyx^{-1}}(\Gamma_x(h^{-1}))  $
\end{proof}

Now, consider the expression 
\begin{align*}
(ht(x)\cdot kt(y))\star lt(z) &= (h\sigma_x(k) t(xy) ) \star lt(z) \\&= hl\sigma_x(k) \Gamma_{xy}(l) \sigma_{(xy) \star z}(h^{-1}l^{-1}\sigma_x(k)^{-1}\Gamma_{z}(h^{-1}\sigma_x(k)^{-1})) \\&  h(xy ,z) t(xy \star z) \tag*{(1)}
\end{align*}
On the other hand 
\begin{align*}
(ht(x)\cdot kt(y))\star lt(z)&=^{ht(x)}(kt(y)\star lt(z)) \cdot (ht(x)\star lt(z) \\&= (ht(x) (kl \Gamma_y(l) \sigma_{(y\star z)}(k^{-1}l^{-1}\Gamma_z(k^{-1}))h(y ,z )t(y \star z)) t(x)^{-1}h^{-1} ) \\& \cdot (hl \Gamma_x(l)\sigma_{(x\star z)}(h^{-1}l^{-1}\Gamma_z(h^{-1}))h(x ,z )t(x \star z)) \\& = h\sigma_x (kl \Gamma_y(l) \sigma_{(y\star z)}(k^{-1}l^{-1}\Gamma_z(k^{-1}))h(y ,z )) \sigma_ {^x(y \star z)}(h^{-1}) t(^x(y \star z)) \\&  \cdot  (hl \Gamma_x(l)\sigma_{(x\star z)}(h^{-1}l^{-1}\Gamma_z(h^{-1})) h(x ,z )t(x \star z)) \\& =  h\sigma_x ( kl \Gamma_y(l) \sigma_{(y\star z)}(k^{-1}l^{-1}\Gamma_z(k^{-1})) h(y ,z ))  \sigma_ {^x(y \star z)} (l \Gamma_x(l) \\&  \sigma_{(x\star z)}(h^{-1}l^{-1}\Gamma_z(h^{-1}))  h(x ,z ))  t(^x(y \star z)) t(x \star z)) \tag*{(2)}
\end{align*}

From equations $(1)$ and $(2)$, we have \\

$l \Gamma_{x}(l) \sigma_{(xy) \star z}(h^{-1}l^{-1}\sigma_x(k^{-1})\Gamma_{z}(h^{-1}\sigma_x(k^{-1})) h(xy ,z) = \sigma_x ( l \sigma_{(y\star z)}(k^{-1}l^{-1}\Gamma_z(k^{-1}))\\ h(y ,z )) \sigma_ {^x(y \star z)} (l \Gamma_x(l)\sigma_{(x\star z)}(h^{-1}l^{-1}\Gamma_z(h^{-1}))h(x ,z )) $ \hspace{6cm}(3)\\

Now consider the expression 
\begin{align*}
ht(x)\star (kt(y)\cdot lt(z)) &=  ht(x)\star (k\sigma_y(l) t(yz) ) \\&=  hk\sigma_y(l) \Gamma_x(k\sigma_y(l)) \sigma_{x \star (yz)}(h^{-1}k^{-1}\sigma_y(l)^{-1}\Gamma_{yz}(h^{-1}))\\& h(x ,yz) t(x \star yz)
\tag*{(4)}
\end{align*}

On the other hand 
\begin{align*}
ht(x)\star (kt(y)\cdot lt(z)) &= (ht(x)\star kt(y)) \cdot ^{kt(y)}(ht(x)\star lt(z)) \\&=
(hk \Gamma_x(k)\sigma_{(x\star y)}(h^{-1}k^{-1}\Gamma_y(h^{-1}))h(x ,y )t(x \star y)) \\&\cdot (kt(y) (hl \Gamma_x(l) \sigma_{(x\star z)}(h^{-1}l^{-1}\Gamma_z(h^{-1}))h(x ,z )t(x \star z)) t(y)^{-1}k^{-1} ) \\& = (hk \Gamma_x(k)\sigma_{(x\star y)}(h^{-1}k^{-1}\Gamma_y(h^{-1}))h(x ,y )t(x \star y)) \\& \cdot (k\sigma_y (hl \Gamma_x(l) \sigma_{(x\star z)}(h^{-1}l^{-1}\Gamma_z(h^{-1}))h(x ,z )t(y)t(x \star z)) t(y)^{-1}k^{-1} ) \\& =  (hk \Gamma_x(k)\sigma_{(x\star y)}(h^{-1}k^{-1}\Gamma_y(h^{-1}))h(x ,y )t(x \star y)) \\& \cdot (k\sigma_y (hl \Gamma_x(l) \sigma_{(x\star z)}(h^{-1}l^{-1}\Gamma_z(h^{-1}))h(x ,z )t(y)t(x \star z)) t(y)^{-1}k^{-1} ) \\&  = (hk \Gamma_x(k)\sigma_{(x\star y)}(h^{-1}k^{-1}\Gamma_y(h^{-1}))h(x ,y ))  \cdot \sigma_{x \star y} ((k\sigma_y (hl \Gamma_x(l) \\& \sigma_{(x\star z)}(h^{-1}l^{-1}\Gamma_z(h^{-1}))h(x ,z )) \sigma_{x \star (yz)}(k^{-1}) t( (x \star y)^y(x \star z)) \tag*{(5)}
\end{align*}

From equations $(4)$ and $(5)$, we have \\

$\sigma_y(l) \Gamma_x(\sigma_y(l)) \sigma_{x \star (yz)}(h^{-1}\sigma_y(l^{-1})\Gamma_{yz}(h^{-1})  h(x ,yz) = \sigma_{(x\star y)}(h^{-1}\Gamma_y(h^{-1}))h(x ,y )\\ \sigma_{(x \star y)y} ( (hl \Gamma_x(l) \sigma_{(x\star z)}(h^{-1}l^{-1}\Gamma_z(h^{-1}))h(x ,z )) $ \hspace{7cm}(6)
\\

Consider the expressions,
\begin{align*}
((ht(x)\star kt(y)) \star ^{kt(y)}lt(z)) & = (ht(x)\star kt(y)) \star (t(y)lt(z)t(y)^{-1}k^{-1}) \\& \hspace*{-2cm} = hk \Gamma_x(k)\sigma_{(x\star y)}(h^{-1}k^{-1}\Gamma_y(h^{-1}))h(x ,y )t(x \star y) \star k\sigma_y(l) \sigma_{^yz}(k^{-1}) t(^yz) 
\\& \hspace*{-2cm}= hk \Gamma_x(k)\sigma_{(x\star y)}(h^{-1}k^{-1}\Gamma_y(h^{-1}))h(x ,y )  k\sigma_y(l)\sigma_{^yz}(k^{-1}) \Gamma_{(x\star y)}(k\sigma_y(l) \\& \hspace*{-1.5cm} \sigma_{^yz}(k^{-1})) \sigma_{((x\star y) \star ^yz )} (h^{-1}k^{-1} \Gamma_x(k)^{-1} \sigma_{(x\star y)}(hk\Gamma_y(h)) h(x ,y )^{-1} k^{-1} \\& \hspace*{-1.5cm}  \sigma_y(l^{-1})\sigma_{^yz}(k)\Gamma_{^yz}(h^{-1}k^{-1} \Gamma_x(k)^{-1} \sigma_{(x\star y)}(hk\Gamma_y(h)) h(x ,y )^{-1}))\\&\hspace*{-1.5cm} h(x \star y, ^yz ) t((x \star y)\star ^yz) 
\end{align*}
Thus, we have 

$ ((ht(x)\star kt(y)) \star ^{kt(y)}lt(z)) = hk^2\Gamma_x(k) \sigma_y(l)\sigma_{^yz}(k^{-1}) h(x ,y ) \Gamma_{(x\star y)}(k\sigma_y(l)\sigma_{^yz}(k^{-1})) $

 \hspace{1cm} $ \sigma_{(x\star y)}(h^{-1}k^{-1}\Gamma_y(h^{-1}))   \sigma_{((x\star y) \star ^yz )} (h^{-1}k^{-2} \Gamma_x(k)^{-1} \sigma_{(x\star y)}(hk\Gamma_y(h)) h(x ,y )^{-1}  \sigma_y(l^{-1})$
 
 \hspace{1.5cm}$ \sigma_{^yz}(k) \Gamma_{^yz}(h^{-1}k^{-1} \Gamma_x(k)^{-1} \sigma_{(x\star y)}(hk\Gamma_y(h)) h(x ,y )^{-1})) h(x \star y, ^yz ) t((x \star y)\star ^yz) $\\

Similarly, we can calculate  

 $((kt(y)\star lt(z)) \star ^{lt(z)}ht(x)) \ \text{and} \ ((lt(z)\star ht(x)) \star ^{ht(x)}kt(y)). $\\
 
Since
$ ((ht(x)\star kt(y)) \star ^{kt(y)}lt(z))((kt(y)\star lt(z)) \star ^{lt(z)}ht(x)) ((lt(z)\star ht(x)) \star ^{ht(x)}kt(y)) = 1 $, we have the following equation\\

$( hk^2\Gamma_x(k) \sigma_y(l)\sigma_{^yz}(k^{-1}) h(x ,y ) \Gamma_{(x\star y)}(k\sigma_y(l)\sigma_{^yz}(k^{-1}))\sigma_{(x\star y)}(h^{-1}k^{-1}\Gamma_y(h^{-1}))\\   \sigma_{((x\star y) \star ^yz )} (h^{-1}k^{-2} \Gamma_x(k)^{-1} \sigma_{(x\star y)}(hk\Gamma_y(h)) h(x ,y )^{-1}  \sigma_y(l^{-1})\sigma_{^yz}(k)  \Gamma_{^yz}(h^{-1}k^{-1} \Gamma_x(k)^{-1} \\ \sigma_{(x\star y)}(hk\Gamma_y(h)) h(x ,y )^{-1})) h(x \star y, ^yz ) \sigma_{((x \star y)\star ^yz))} ((kl^2\Gamma_y(l) \sigma_z(h)\sigma_{^zx}(l^{-1}) h(y ,z ) \Gamma_{(y\star z)}(l\sigma_z(h) \\ \sigma_{^zx}(l^{-1})) \sigma_{(y\star z)}(k^{-1}l^{-1}\Gamma_z(k^{-1}))   \sigma_{((y\star z) \star ^zx )} (k^{-1}l^{-2} \Gamma_y(l)^{-1} \sigma_{(y\star z)}(kl\Gamma_z(k)) h(y ,z )^{-1}  \sigma_z(h^{-1})\sigma_{^zx}(l) \\ \Gamma_{^zx}(k^{-1}l^{-1} \Gamma_y(l)^{-1}\sigma_{(y\star z)}(kl\Gamma_z(k)) h(y ,z )^{-1})) h(y \star z, ^zx )) \sigma_{((y \star z)\star ^zx) )((z \star x)\star ^xy))}  ((lh^2\Gamma_z(h) \sigma_x(k) \\ \sigma_{^xy}(h^{-1}) h(z ,x )  \Gamma_{(z\star x)}(h\sigma_x(k)\sigma_{^xy}(h^{-1}))\sigma_{(z\star x)}(l^{-1}h^{-1}\Gamma_x(l^{-1}))   \sigma_{((z\star x) \star ^xy )} (l^{-1}h^{-2} \Gamma_z(h)^{-1} \\ \sigma_{(z\star x)}  (lh\Gamma_x(l)) h(z ,x )^{-1}  \sigma_x(k^{-1})\sigma_{^xy}(h)  \Gamma_{^xy}(l^{-1}h^{-1} \Gamma_x(h)^{-1}\sigma_{(z\star x)}(lh\Gamma_x(l)) h(z ,x )^{-1}))\\ h(z \star x, ^xy ))  = 1 $ \hspace{12cm}(7)\\

Now consider the expression 
\begin{align*}
^{lt(z)}(ht(x)\star kt(y)) &=lt(z) hk \Gamma_x(k)\sigma_{(x\star y)}(h^{-1}k^{-1}\Gamma_y(h^{-1}))h(x ,y )t(x \star y) t(z^{-1})l^{-1}
\\& =l\sigma_z (hk \Gamma_x(k)\sigma_(x\star y)(h^{-1}k^{-1}\Gamma_y(h^{-1}))h(x ,y )) \sigma_{^z(x\star y)}(l^{-1})  t(^z(x\star y)) \tag*{(8)}
\end{align*}
Also, we have
\begin{align*}
^{lt(z)}ht(x)\star ^{lt(z)} kt(y) &=  lt(z)ht(x)t(z^{-1})l^{-1}\star {lt(z)}kt(y)t(z^{-1})l^{-1} \\& = (l\sigma _z(h)\sigma_{^zx}(l^{-1})t(^zx))\star (l\sigma _z(k)\sigma_{^zy}(l^{-1})t(^zy)) \\& =  l^2\sigma _z(h) \sigma _z(k) \sigma_{^zx}(l^{-1}) \sigma_{^zy}(l^{-1}) \Gamma_{^zx}(l\sigma _z(k)\sigma_{^zy}(l^{-1}))  \sigma_{^z(x\star y)}(l^{-2} \sigma _z(h^{-1})\\& \sigma _z(k^{-1}) \sigma_{^zx}(l)
\sigma_{^zy}(l) \Gamma_{^zy}(l^{-1}\sigma _z(h^{-1})\sigma_{^zx}(l) ))  h(^zx , ^zy) t(^zx \star ^zy)  \tag*{(9)} 
\end{align*}
From equations $(8)$ and $(9)$, we have
\\

$ \sigma_z ( \Gamma_x(k)\sigma_{(x\star y)}(h^{-1}k^{-1}\Gamma_y(h^{-1}))h(x ,y )) = l \sigma_{^zx}(l^{-1}) \sigma_{^zy}(l^{-1}) \Gamma_{^zx}(l\sigma _z(k)\sigma_{^zy}(l^{-1}))\\ \sigma_{^z(x\star y)}(l^{-1} \sigma _z(h^{-1}) \sigma _z(k^{-1}) \sigma_{^zx}(l) \sigma_{^zy}(l) \Gamma_{^zy}(l^{-1}\sigma _z(h^{-1})\sigma_{^zx}(l) )) h(^zx , ^zy) $	\hspace{2.55cm}(10)

\begin{lemma} \label{commutes} If  $K$ is an abelian group, then 	$\sigma_x \circ \Gamma_z = \Gamma_z \circ \sigma_x  , \forall  x,z\in K $. 
	
\end{lemma}
\begin{proof} Let $h\in H$. Then 

$(  \sigma_x \circ \Gamma_z)(h) = \sigma_x(t(z)\star h)  = t(x)(t(z)\star h)t(x)^{-1} = ~ ^{t(x)}(t(z)\star h)  $ \hspace{2cm}(11)\\
	
On the other hand, \\

	$(\Gamma_z \circ \sigma_x)(h)=   \Gamma_z(t(x)ht(x)^{-1})=   t(z)\star (t(x)ht(x)^{-1})
	=t(z)\star (^{t(x)}h)$ 
	
\hspace{2.2cm}	$= ^{t(x)}(^{t(x)^{-1}}t(z)\star h) = ^{t(x)}(t(x^{-1}zx)\star h) = ^{t(x)}(t(z)\star h) $ \hspace{2cm}(12)\\

	By equation $(11)$ and $(12)$, we have
	
	$\Gamma_z \circ \sigma_x =  \sigma_x \circ \Gamma_z. $ That is, $\Gamma_z$ and  $\sigma_x $ commutes with each other.
\end{proof}

From the above discussion, we have the following theorem:
\begin{theorem} \label{semi-direct}
Let $G = H \rtimes_{\sigma} K,$ where $H$ is an abelian group with trivial multiplicative Lie algebra structure and $K$ is a group. Suppose $\star$ is a multiplicative Lie algebra structure on $K$, and  maps $\Gamma : K \to End(H)$ and  $ h : K\times K \to H$ that satisfies the following conditions for all $x, y, z \in K$ and   	$h, k, l \in H:$
\begin{enumerate}
	\item $ h(x,1)=h(1,x)=h(x,x)=1$;
	\item $\Gamma_{xy}(h)=\Gamma_x(h)\sigma_x (\Gamma_y(h)) $ and $\Gamma_{x\star y}(\sigma_y(h))=\Gamma_x (\Gamma_y(h)) \Gamma_{xyx^{-1}} (\Gamma_x(h^{-1})) $;
	\item $l \Gamma_{x}(l) \sigma_{(xy) \star z}(h^{-1}l^{-1}\sigma_x(k^{-1})\Gamma_{z}(h^{-1}\sigma_x(k^{-1})) h(xy ,z) = \sigma_x ( l  \sigma_{(y\star z)}(k^{-1}l^{-1}\Gamma_z(k^{-1}))\\ h(y ,z )) \sigma_ {^x(y \star z)} (l \Gamma_x(l)\sigma_{(x\star z)}(h^{-1}l^{-1}\Gamma_z(h^{-1}))h(x ,z )) $;
	\item $\sigma_y(l) \Gamma_x(\sigma_y(l)) \sigma_{x \star (yz)}(h^{-1}\sigma_y(l^{-1})\Gamma_{yz}(h^{-1})  h(x ,yz) = \sigma_{(x\star y)}(h^{-1}\Gamma_y(h^{-1}))h(x ,y )\\ \sigma_{(x \star y)y} ( (hl \Gamma_x(l) \sigma_{(x\star z)}(h^{-1}l^{-1}\Gamma_z(h^{-1}))h(x ,z )) $;
	\item $ ( hk^2\Gamma_x(k) \sigma_y(l)\sigma_{^yz}(k^{-1}) h(x ,y ) \Gamma_{(x\star y)}(k\sigma_y(l)\sigma_{^yz}(k^{-1}))\sigma_{(x\star y)}(h^{-1}k^{-1}\Gamma_y(h^{-1}))\\   \sigma_{((x\star y) \star ^yz )} (h^{-1}k^{-2} \Gamma_x(k)^{-1} \sigma_{(x\star y)}(hk\Gamma_y(h)) h(x ,y )^{-1}  \sigma_y(l^{-1})\sigma_{^yz}(k)  \Gamma_{^yz}(h^{-1}k^{-1} \Gamma_x(k)^{-1} \\ \sigma_{(x\star y)}(hk\Gamma_y(h)) h(x ,y )^{-1})) h(x \star y, ^yz ) \sigma_{((x \star y)\star ^yz))} ((kl^2\Gamma_y(l) \sigma_z(h)\sigma_{^zx}(l^{-1}) h(y ,z )\\ \Gamma_{(y\star z)}(l\sigma_z(h)  \sigma_{^zx}(l^{-1})) \sigma_{(y\star z)}(k^{-1}l^{-1}\Gamma_z(k^{-1}))   \sigma_{((y\star z) \star ^zx )} (k^{-1}l^{-2} \Gamma_y(l)^{-1} \sigma_{(y\star z)}(kl\Gamma_z(k)) \\ h(y ,z )^{-1}  \sigma_z(h^{-1})\sigma_{^zx}(l)  \Gamma_{^zx}(k^{-1}l^{-1} \Gamma_y(l)^{-1}\sigma_{(y\star z)}(kl\Gamma_z(k)) h(y ,z )^{-1})) h(y \star z, ^zx )) \\\sigma_{((y \star z)\star ^zx) )((z \star x)\star ^xy))}  ((lh^2\Gamma_z(h) \sigma_x(k)  \sigma_{^xy}(h^{-1}) h(z ,x )  \Gamma_{(z\star x)}(h\sigma_x(k)\sigma_{^xy}(h^{-1})) \\ \sigma_{(z\star x)}(l^{-1}h^{-1}\Gamma_x(l^{-1}))   \sigma_{((z\star x) \star ^xy )} (l^{-1}h^{-2} \Gamma_z(h)^{-1}  \sigma_{(z\star x)}  (lh\Gamma_x(l)) h(z ,x )^{-1}  \sigma_x(k^{-1}) \\ \sigma_{^xy}(h)  \Gamma_{^xy}(l^{-1}h^{-1} \Gamma_x(h)^{-1}\sigma_{(z\star x)}(lh\Gamma_x(l)) h(z ,x )^{-1})) h(z \star x, ^xy ))  = 1 $;
	\item $ \sigma_z ( \Gamma_x(k)\sigma_{(x\star y)}(h^{-1}k^{-1}\Gamma_y(h^{-1}))h(x ,y )) = l \sigma_{^zx}(l^{-1}) \sigma_{^zy}(l^{-1}) \Gamma_{^zx}(l\sigma _z(k)\sigma_{^zy}(l^{-1}))\\ \sigma_{^z(x\star y)}(l^{-1} \sigma _z(h^{-1}) \sigma _z(k^{-1}) \sigma_{^zx}(l) \sigma_{^zy}(l) \Gamma_{^zy}(l^{-1}\sigma _z(h^{-1})\sigma_{^zx}(l) )) h(^zx , ^zy) .$
\end{enumerate}	
Then we have a multiplicative Lie algebra structure $\tilde{\star}$ on $G$ defined  by 
$$(h,x)\tilde{\star} (k,y) = (hk\Gamma_x(k)\sigma_{(x\star y)}(h^{-1}k^{-1}\Gamma_y(h^{-1})) h(x,y), x \star y) ~\forall~ (h, x), (k, y)\in G.$$	In this case, we call $\tilde{\star}$ is induced by  $\star$, and maps $\Gamma$ and $h$.

Conversely, let $\tilde{\star}$ be a  multiplicative Lie algebra structure  on $G$ such that $H$ is an ideal. Then there is a multiplicative Lie algebra structure $\star$ on $K$, and  maps $\Gamma : K \to End(H)$ and  $ h : K\times K \to H$ that satisfies  conditions from (1) to (6) given above such that $(h,x)\tilde{\star} (k,y) = (hk\Gamma_x(k)\sigma_{(x\star y)}(h^{-1}k^{-1}\Gamma_y(h^{-1})) h(x,y), x \star y) ~\forall~ (h, x), (k, y)\in G.$ In this case, we call $\tilde{\star}$ is determined by  $\star$, and maps $\Gamma$ and $h$.

%Furthermore, if $G$ has a multiplicative Lie algebra structure $\tilde{\star} $ such that $	H$ is ideal of $G$ and $\frac{G}{H} \cong (K,\cdot,\star),$ then the multiplicative Lie algebra structure $\tilde{\star} $  on $G$ must have arisen in this way. In this case, we say that we have a multiplicative Lie algebra structure $\tilde{\star}$ on $G$  must have arisen by a multiplicative Lie algebra structure $\star$ on $K$ in the same way.
	\end{theorem}

\begin{remark} \label{direct}
In particular, suppose  $G = H \times K$ (that is, $\sigma = I_H$) and there are maps $\Gamma : K \to End(H)$ and  $ h : K\times K \to H$ that satisfies the following conditions for all $x, y, z \in K$  and  $h, k, l \in H:$
\begin{enumerate}
	\item $ h(x,1)=h(1,x)=h(x,x)=1  $;
	
	\item $\Gamma_{xy}(h)=\Gamma_x(h)\Gamma_y(h) $ and $\Gamma_{x\star y}(h)=\Gamma_x (\Gamma_y(h)) \Gamma_{y} (\Gamma_x(h^{-1})), $ that is, $\Gamma$ is a multiplicative Lie algebra homomorphism;
	
	\item $  h(xy ,z) =  h(x ,z )h(y ,z )  $;
	
	\item $   h(x ,yz) = h(x ,y )h(x ,z ) $;
	
	\item $  \Gamma_{(x\star y)}(l)  \Gamma_{(y\star z)}(h)  \Gamma_{(z\star x)} (k) \Gamma_z(h(x, y)^{-1}) \Gamma_x(h(y, z)^{-1}) \Gamma_y(h(z, x)^{-1})  h(x \star y, ^yz )\\ h(y \star z, ^zx ) h(z \star x, ^xy ) = 1 $;
	\item $ h(^zx , ^zy) = h(x ,y ). $
\end{enumerate}
Then we have a multiplicative Lie algebra structure $\tilde{\star}$ on $G$ defined  by
$$(h,x)\tilde{\star} (k,y) = ( \Gamma_x(k)\Gamma_y(h^{-1})h(x ,y ), x \star y) ~\forall~ (h, x), (k, y)\in G.$$
Conversely, let $\tilde{\star}$ be a  multiplicative Lie algebra structure  on $G$ such that $H$ is an ideal. Then there is a multiplicative Lie algebra structure $\star$ on $K$, and  maps $\Gamma : K \to End(H)$ and  $ h : K\times K \to H$ that satisfies  conditions from (1) to (6) given above such that $$(h,x)\tilde{\star} (k,y) = ( \Gamma_x(k)\Gamma_y(h^{-1})h(x ,y ), x \star y) ~\forall~ (h, x), (k, y)\in G.$$

% Furthermore, $G$ has a multiplicative Lie algebra structure  $\tilde{\star} $ such that $	H$ is ideal of $G$ and $\frac{G}{H} \cong (K,\cdot,\star), $ then the multiplicative Lie algebra structure $\tilde{\star} $  on $G$ must have arisen in this way.

% Then we have a multiplicative lie algebra structure $\tilde{\star} $ on $G$ such that $	H$ is ideal of $G$ and $\frac{G}{H} \cong (K,\cdot,\star) $  if and only if we have maps $\Gamma : K \to End(H)$ and  $ h : K\times K \to H,$ which satisfy the following equations: 
%\begin{enumerate}
%	\item $ h(x,1)=h(1,x)=h(x,x)=1  $
%	
%	\item $\Gamma_{xy}(h)=\Gamma_x(h)\sigma_x (\Gamma_y(h)) $ and $\Gamma_{x\star y}(h)=\Gamma_x (\Gamma_y(h)) \Gamma_{y} (\Gamma_x(h^{-1})), $ that is, $\Gamma$ is multiplicative lie algebra homomorphism.
%	
%	\item $  h(xy ,z) =  h(x ,z )h(y ,z )  $
%	
%	\item $   h(x ,yz) = h(x ,y )h(x ,z ) $
%	
%	\item $  \Gamma_{(x\star y)}(l)  \Gamma_{(y\star z)}(h)  \Gamma_{(z\star x)} (k) \Gamma_z(h(x, y)^{-1}) \Gamma_x(h(y, z)^{-1}) \Gamma_y(h(z, x)^{-1})  h(x \star y, ^yz )\\ h(y \star z, ^zx ) h(z \star x, ^xy ) = 1 $
%	\item $ h(^zx , ^zy) = h(x ,y ) $
%\end{enumerate}
%and it is defined by  
%
%$(h,x)\tilde{\star} (k,y) = ( \Gamma_x(k)\Gamma_y(h^{-1})h(x ,y ), x \star y)$	
\end{remark}
\begin{remark} \label{gcd direct}
	Let $G = H\times K,$ where $H$ is an abelian group with trivial multiplicative Lie algebra structure and $K$ is a finite group generated by two elements $a$ and $b$ such that $(|H|, |a|) = 1.$ It is easy to verify that there is no non-trivial alternating map from $K \times K$ to $H$. Therefore, every multiplicative Lie algebra structure $\tilde{\star} $ on $G$ with respect to which $H$ is an ideal is determined by  a multiplicative Lie algebra structure $\star$ on $K$ and a multiplicative Lie algebra homomorphism $\Gamma : K \to End(H)$, and it is defined by 
	$$(h,x)\tilde{\star} (k,y) = ( \Gamma_x(k)\Gamma_y(h^{-1}), x \star y),$$ where $\Gamma$ satisfies 
		$  \Gamma_{(x\star y)}(l)  \Gamma_{(y\star z)}(h)  \Gamma_{(z\star x)} (k)  = 1 $ for all $x, y, z \in K$ and $h, k, l \in H.$ 	
	
	\end{remark}
\begin{proposition} \label{gcd}
Let $G = H\times K,$ where $H$ is an abelian group of order $m$ with trivial multiplicative Lie algebra structure and $K$ is a group of order $n$ such that $(m, n) = 1.$ Then every multiplicative Lie algebra structure $\tilde{\star} $ on $G$ is determined by a multiplicative Lie algebra structure $\star$ on $K$ and map $\Gamma : K \to End(H),$  that satisfies the following conditions for all $x, y, z \in K$ and $h, k, l \in H:$ 	
\begin{enumerate}

		\item $\Gamma_{xy}(h)=\Gamma_x(h)\Gamma_y(h) $ and $\Gamma_{x\star y}(h)=\Gamma_x (\Gamma_y(h)) \Gamma_{y} (\Gamma_x(h^{-1}));$
		
		\item $  \Gamma_{(x\star y)}(l)  \Gamma_{(y\star z)}(h)  \Gamma_{(z\star x)} (k)  = 1 $
	
	\end{enumerate}	
and it is defined by 
 
$(h,x)\tilde{\star} (k,y) = ( \Gamma_x(k)\Gamma_y(h^{-1}), x \star y).$
	\end{proposition}
	
%	 Let $H$ be a abelian group of order $m$ with trivial multiplicative Lie algebra structure and $K$ be a group of order $n$. Then every multiplicative Lie algebra structure $\tilde{\star} $ on $G = H\times K,$ with $(m, n) = 1,$ is determined by a multiplicative Lie algebra structure $\star$ on $K$ and maps $\Gamma : K \to End(H)$ and $ h : K\times K \to H,$ which satisfy the following equations:
%\begin{enumerate}
%	\item $ h(x,1)=h(1,x)=h(x,x)=1  $
%	
%	\item $\Gamma_{xy}(h)=\Gamma_x(h)\sigma_x (\Gamma_y(h)) $ and $\Gamma_{x\star y}(h)=\Gamma_x (\Gamma_y(h)) \Gamma_{y} (\Gamma_x(h^{-1})), $ that is, $\Gamma$ is multiplicative lie algebra homomorphism.
%	
%	\item $  h(xy ,z) =  h(x ,z )h(y ,z )  $
%	
%	\item $   h(x ,yz) = h(x ,y )h(x ,z ) $
%	
%	\item $  \Gamma_{(x\star y)}(l)  \Gamma_{(y\star z)}(h)  \Gamma_{(z\star x)} (k) \Gamma_z(h(x, y)^{-1}) \Gamma_x(h(y, z)^{-1}) \Gamma_y(h(z, x)^{-1})  h(x \star y, ^yz )\\ h(y \star z, ^zx ) h(z \star x, ^xy ) = 1 $
%	\item $ h(^zx , ^zy) = h(x ,y ) $
%\end{enumerate}	
%and it is defined by  
%
%$(h,x)\tilde{\star} (k,y) = ( \Gamma_x(k)\Gamma_y(h^{-1})h(x ,y ), x \star y)$

\begin{proof} Let $\tilde{\star}$ be a multiplicative Lie algebra structure on $G$.  By Remark \ref{direct}, it is sufficient to show that $H$ is an ideal of $G$ and there is only trivial bilinear map  $ h $ from $ K\times K $  to $ H.$ 
	
Let $a\in H$ and $g \in G$. Then 
 $1 = a^m \tilde{\star} g = (a\tilde{\star} g)^m.$ Suppose $a\tilde{\star} g = hy, $ where $h \in H $ and $y\in K.$ Since $H \subseteq Z(G), 1  = (a\tilde{\star} g)^m = h^my^m = y^m.$ This implies that , $y = 1.$ Hence, $a\tilde{\star} g \in H , \forall g\in G.$ 
 
 Suppose $x,y \in K$ and order of $y$ is $k.$ Then $h(x ,y^k )= 0 = h(x ,y )^k.$ Since $(m, k) = 1 $, we have $h(x,y) = 0.$  Now, it is easy to see that $h(a, b) = 0$ for all $a, b\in  K $.    
\end{proof}

\begin{example} Let $D_p$ be the dihedral group of order $2p$ with multiplicative Lie algebra structure $  \star. $ Suppose $G = \Z_p \times D_p.$ Then by Remark \ref{direct}, we have a  multiplicative Lie algebra structure $\tilde{\star}$ on $G$ induced by $\star,$ and maps $\Gamma$ and $h,$ and it is defined as $$(h,x)\tilde{\star} (k,y) = ( \Gamma_x(k)\Gamma_y(h^{-1})h(x ,y ), x \star y).$$ 
 It is clear that $h$ is  a bilinear map.  Suppose $x, y\in D_p,$ where order of $y$ is $2.$ Then $h(x ,y^2 )= 0 = h(x ,y )^2.$ Since $\Z_p$ has no element of order $2$, we have $h(x,y)=0$. Now, it is easy to see that $h(a, b) = 0$ for all $a, b\in  D_p$.
 
 Since there is only trivial homomorphism  $D_p \to End(\Z_p)\cong \Z_p,$ $\Gamma$ is trivial. Hence,  $(h,x)\tilde{\star} (k,y) = (0 , x \star y).$ Since $D_p$ has only two multiplicative Lie algebra structures, $G$ has also two multiplicative Lie algebra structures for which $\Z_p$ is ideal.
\end{example}

We already know  that the symmetric group $S_3$ has two distinct multiplicative Lie algebra structure \cite{MS}. The following example give another method to compute the same with the help of Theorem \ref{semi-direct}.  
  
\begin{example} Let $G = \Z_3 \rtimes_\sigma \Z_2$ and $\tilde{\star}$ be a non-trivial multiplicative Lie algebra structure on $G$, where $\sigma: \Z_2 \to Aut(\Z_3)$ is non-trivial group homomorphism.  Since $\Z_3$ is the only proper normal subgroup of $G$, $G\tilde{\star} G = \Z_3$. Now, by Theorem \ref{semi-direct},  $\tilde{\star}$ determined by a multiplicative Lie algebra structure $\star $ on $\Z_2,$ and maps $\Gamma$ and $h.$ Since $\Z_2$ has only trivial multiplicative Lie algebra structures,  
		$$(h,x)\tilde{\star} (k,y) = (\Gamma_x(k)\Gamma_y(h^{-1}))h(x ,y ), 0).$$
It is easy to see that $h(x,y)=0$ for all $x,y \in \Z_2.$ Hence, $$(h,x)\tilde{\star} (k,y) = (\Gamma_x(k)\Gamma_y(h^{-1})), 0).$$		
%Since $\Z_2$ and $\Z_3$ have only trivial multiplicative Lie algebra structures, $\Gamma:\Z_2 \to End(\Z_3)\cong \Z_3$ is a group homomorphism. Suppose $y\in \Z_2 $ $o(y)=2.$ Then $h(x ,y^2 )= 0 = h(x ,y )^2.$ In this case $h(x,y)=0$ for all $x,y \in \Z_2.$  Therefore 
%	$(h,x)\tilde{\star} (k,y) = ( \Gamma_x(k)\Gamma_y(h^{-1}), 0).$ Also
Also, it is easy to see that there is only one  non zero map  $\Gamma:\Z_2 \to End(\Z_3)$ which satisfies $\Gamma_{xy}(h)=\Gamma_x(h)\sigma_x (\Gamma_y(h))$ for every $h \in \Z_3$.		
\end{example}
%Similarly, Theorem \ref{semi-direct} is useful to compute distinct multiplicative Lie algebra structures on other groups also. 

\begin{example} Consider the group $G = \Z_p \times D_n, (p,2n)=1.$ Then by Proposition \ref{gcd}, every multiplicative Lie algebra structure $\tilde{\star}$ on $G$ is determined by a multiplicative Lie algebra structure $\star$ on $D_n$ and map $\Gamma$. Since there is only trivial homomorphism $D_n \to End(\Z_p)\cong \Z_p,$  $\Gamma$ is trivial. Therefore,  	$(h,x)\tilde{\star} (k,y) = ( 0, x \star y).$
	 
Now, by Theorem 2.5 of \cite{MS},  $D_n$ has $\tau(n)$ multiplicative Lie algebra structures. So, $G$ has also $\tau(n)$ multiplicative Lie algebra structures for which $\Z_p$ is ideal.
\end{example}

\begin{example}
	
 Let $Q_n$ be the quaternion group of order $4n.$ Suppose $G = \Z_p \times Q_n, (p, 4n) = 1$. Then by Proposition \ref{gcd}, every multiplicative Lie algebra structure $\tilde{\star}$ on $G$ is determined by a multiplicative Lie algebra structure $\star$ on $Q_n$ and map $\Gamma$. Therefore, $(h,x)\tilde{\star} (k,y) = ( \Gamma_x(k)\Gamma_y(h^{-1}), x \star y).$
  
 Since there is only trivial homomorphism  $Q_n \to End(\Z_p)\cong \Z_p,$ $\Gamma$ is trivial. Hence,  $(h,x)\tilde{\star} (k,y) = (0 , x \star y).$ Now,  by Theorem 2.5 of \cite{MS},  $Q_n$ has $\tau(n)$ multiplicative Lie algebra structures. So, $G$ has also $\tau(n)$ multiplicative Lie algebra structures for which $\Z_p$ is ideal.

%   where $ h:Q_n\times Q_n \to \Z_m$ is a bilinear map satisfying $(4)$ and $(5)$ in Remark \ref{direct}. Suppose $y\in Q_n$ and $o(y)=2.$ Then $h(x ,y^2 )= 0 = h(x ,y )^2.$ In this case $h(x,y)=0$ for all $x,y \in Q_n.$ Therefore, 	$(h,x)\tilde{\star} (k,y) = ( \Gamma_x(k)\Gamma_y(h^{-1}), x \star y).$ Also, there is only trivial homomorphism $\Gamma:Q_n \to End(\Z_m)\cong \Z_m.$ Hence, $(h,x)\tilde{\star} (k,y) = ( 0 , x \star y).$ Since, by Theorem 2.5 of \cite{MS},  $Q_n$ has $\tau(n)$ multiplicative Lie algebra structures, $G$ has $\tau(n)$ multiplicative Lie algebra structures.
\end{example}

\begin{example}Let $D_4 = \langle a, b \mid a^2 = 1 = b^4 = 1, ab = b^{-1}a\rangle$ be the dihedral group of order $8$ with  multiplicative Lie algebra structure $\star$.
Suppose $G = \Z_4 \times D_4.$ Let $\Gamma: D_4 \to End(\Z_4) = \{\tilde{0}, \tilde{1}, \tilde{2}, \tilde{3}\}$  and $h: D_4 \times D_4 \to \Z_4$ be maps satisfying all  conditions given in Remark \ref{direct}. Hence,  we have a  multiplicative Lie algebra structure $\tilde{\star}$ on $G$  defined as  	$$(h,x)\tilde{\star} (k,y) = ( \Gamma_x(k)\Gamma_y(h^{-1})h(x ,y ), x \star y).$$  
Suppose $x, y\in D_4,$ where order of $y$ is $2.$ Then $h(x ,y^2 )= \bar{0} = h(x ,y )^2, ~\text{order of }~ h(x ,y )$ is either $1 \ \text{or} \ 2. $ Hence, there are only two bilinear maps, one is trivial and the other one is defined by $h(a,b)= \bar{2}$.  Also, there are four group homomorphism  
$\Gamma $ from $D_4$ to $End(\Z_4) = \{\tilde{0}, \tilde{1}, \tilde{2}, \tilde{3}\} $ defined by
\begin{enumerate}
\item  $ \Gamma_{a} = \tilde{0}  $ and  $ \Gamma_{b} = \tilde{0}  $
		
\item $ \Gamma_{a} = \tilde{2}  $ and  $ \Gamma_{b} = \tilde{0}  $
		
\item $ \Gamma_{a} = \tilde{0}  $ and  $ \Gamma_{b} = \tilde{2}  $
		
\item $ \Gamma_{a} = \tilde{2}  $ and  $ \Gamma_{b} = \tilde{2}  $
\end{enumerate}	

We know that $D_4$ has three distinct multiplicative Lie algebra structures defined as $ a \star b = 1, a \star b = b \ \text{and} \ a \star b = b^2 = [a , b]  $ (Theorem 2.5, \cite{MS}).
	
\textbf{Case I:} For $ a \star b = 1,$ it is easy to see that every pair of $(\Gamma, h)$ satisfies all  the conditions given in Remark \ref{direct}, where $\Gamma: D_4 \to End(\Z_4)$ is a group homomorphism  and $h: D_4 \times D_4 \to \Z_4$ is a bilinear map.

If $\tilde{\star}$ is non trivial, then $G \tilde{\star} G \cong \Z_2 . $
	
	\textbf{Case II:} For $ a \star b = b,$ there are two multiplicative Lie algebra homomorphism  
	$\Gamma $ from $D_4$ to $End(\Z_4) = \{\tilde{0}, \tilde{1}, \tilde{2}, \tilde{3}\} $ given by
	\begin{enumerate}
		\item  $ \Gamma_{a} = \tilde{0}  $ and  $ \Gamma_{b} = \tilde{0}  $
		
		\item $ \Gamma_{a} = \tilde{2}  $ and  $ \Gamma_{b} = \tilde{0}  $
	\end{enumerate}	
Also, these two multiplicative Lie algebra homomorphisms satisfy all the conditions given in Remark \ref{direct} with every bilinear map $h: D_4 \times D_4 \to \Z_4$. 
	
In this case, $G \tilde{\star} G \cong \Z_2 \times \Z_4 ~\text{or}~  \Z_4$
	
	\textbf{Case III:} Similarly, for $ a \star b = b^2,$ it is easy to see that every pair of $(\Gamma, h)$ satisfies all  the conditions given in Remark \ref{direct}, where $\Gamma: D_4 \to End(\Z_4)$ is a group homomorphism  and $h: K \times K \to H$ is a bilinear map.

	In this case, $G \tilde{\star} G \cong \Z_2 \times \Z_2 ~\text{or}~  \Z_2$.
	
\end{example}

%\begin{example}Let $D_4$ be the dihedral group of order $8$ with multiplicative Lie algebra structure $  \star. $ Suppose  $G = \Z_4 \times D_4.$ Then by Remark \ref{direct},   every multiplicative Lie algebra structure $\tilde{\star}$ on $G$ induced by $\star$ is defined as $(h,x)\tilde{\star} (k,y) = ( \Gamma_x(k)\Gamma_y(h^{-1})h(x ,y ), x \star y)$, where $ h:D_4\times D_4 \to \Z_4$ is a bilinear map satisfying $(4)$ and $(5)$ in Remark \ref{direct}. Suppose $y\in D_4$ and $o(y)=2.$ Then $h(x ,y^2 )= 0 = h(x ,y )^2, o(h(x ,y ))=1 \ \text{or} \ 2. $ Let $o(h(x ,y ))=2.$  Then We have a contradiction. Hence, $h(x,y)=0$ for all $x,y \in D_4.$ Therefore, 	$(h,x)\tilde{\star} (k,y) = ( \Gamma_x(k)\Gamma_y(h^{-1}), x \star y).$  Also there are four homomorphism $\Gamma:D_4 \to End(\Z_4)\cong \Z_4$ and by Theorem 2.5 of \cite{MS}, $D_4 $ has three multiplicative Lie algebra structure. So $G$ has atleast $12$ multiplicative Lie algebra structures.  
%\end{example}

\noindent{\bf Acknowledgement:}
The first named author sincerely thanks IIIT Allahabad and Ministry of Education, Government of India for providing institute fellowship. The second named author sincerely thanks IIIT Allahabad and University grant commission (UGC), Govt. of India, New Delhi for research fellowship.

\end{document}